\documentclass[preprint,11pt]{elsarticle}
\makeatletter
\def\ps@pprintTitle{%
 \let\@oddhead\@empty
 \let\@evenhead\@empty
 \def\@oddfoot{\centerline{\thepage}}%
 \let\@evenfoot\@oddfoot}
\makeatother
\usepackage{amssymb,amsmath,amsthm}
\usepackage{epsfig}
\usepackage{color}
\usepackage{makeidx}
\usepackage{bbm}
\usepackage[top=1in, bottom=1in, inner=1in, outer=1in]{geometry}
\usepackage{setspace}

\newtheorem{theorem}{Theorem}

\newtheorem*{propa}{Proposition A}
\newtheorem*{propb}{Proposition B}

\newtheorem{proposition}{Proposition}

\newtheorem{remark}{Remark}
\newcommand*\xbar[1]{%
  \hbox{%
    \vbox{%
      \hrule height 0.5pt % The actual bar
      \kern0.4ex%         % Distance between bar and symbol
      \hbox{%
        \kern-0.15em%      % Shortening on the left side
        \ensuremath{#1}%
        \kern-0.15em%      % Shortening on the right side
      }%
    }%
  }%
}

\begin{document}
\begin{frontmatter}

\title{Branching Brownian motion deactivated at the boundary of an expanding ball}

\author{Mehmet \"{O}z}
\ead{mehmet.oz@ozyegin.edu.tr}
\ead[url]{https://faculty.ozyegin.edu.tr/mehmetoz/}

\author{Elif Aydo\u{g}an}
\ead{elif.aydogan.27792@ozu.edu.tr}

\address{Department of Natural and Mathematical Sciences, Faculty of Engineering, \"{O}zye\u{g}in University, Istanbul, Turkey}

\begin{abstract}

We study a $d$-dimensional branching Brownian motion inside subdiffusively expanding balls, where the boundary of the ball is \emph{deactivating} in the sense that once a particle hits the moving boundary, it is instantly deactivated but is reactivated later if and when its ancestral line becomes fully inside the expanding ball at that later time. That is, at each time, the process consists of the particles of a normal branching Brownian motion whose ancestral lines are fully inside the expanding ball at that time. We obtain a full limit large-deviation result as time tends to infinity on the probability that the mass of the process inside the expanding ball is aytpically small. A phase transition at a critical rate of expansion for the ball is identified, at which the nature of the optimal strategy to realize the large-deviation event changes, and the Lyapunov exponent giving the decay rate of the associated large-deviation probability is continuous. As a corollary, we also obtain a kind of law of large numbers for the mass of the process inside the expanding ball.

\end{abstract}

\vspace{3mm}

\begin{keyword}
Branching Brownian motion \sep large deviations \sep killing boundary 
\vspace{3mm}
\MSC[2020] 60J80 \sep 60F10 \sep 92D25
\end{keyword}

\end{frontmatter}

\pagestyle{myheadings}
\markright{BBM in an expanding ball\hfill}

\section{Introduction}\label{intro}

In this work, we study a model of branching Brownian motion (BBM) in a dynamic restricted domain in $\mathbb{R}^d$, where at each time the process consists of the particles of a normal BBM whose ancestral lines at that time are fully inside an expanding ball. In this sense, the boundary of the ball is \emph{deactivating} so that once a particle of the BBM hits the moving boundary, it is instantly deactivated and otherwise continues its life normally but can be reactivated at a later time provided its ancestral line becomes fully inside the expanding ball at that time. In this way, the moving boundary serves as a mass-suppressing mechanism, and therefore the growth of mass, that is, the number of `active' particles, is slower than that of a normal BBM.

\subsection{Formulation of the problem and background}

Let $Z=(Z_t)_{t\geq 0}$ be a strictly dyadic $d$-dimensional BBM with branching rate $\beta>0$, where $t$ represents time. The process starts with a single particle, which performs a Brownian motion in $\mathbb{R}^d$ for a random lifetime, at the end of which it dies and simultaneously gives birth to two offspring. Similarly, starting from the position where their parent dies, each offspring repeats the same procedure as their parent independently of others and the parent, and the process evolves through time in this way. All particle lifetimes are exponentially distributed with constant parameter $\beta>0$. For each $t\geq 0$, $Z_t$ can be viewed as a finite discrete measure on $\mathbb{R}^d$, which is supported at the positions of the particles at time $t$. 

The model of a \emph{BBM with deactivation at a moving boundary} was introduced in \cite{O2023} as follows. For a Borel set $A\subseteq\mathbb{R}^d$, denote by $\partial A$ the boundary of $A$. Consider a family of Borel sets $B=(B_t)_{t\geq 0}$. We will denote by $Z^B=(Z_t^{B_t})_{t\geq 0}$ a BBM with deactivation at $\partial B$. For each $t\geq 0$, start with $Z_t$, and delete from it any particle whose ancestral line up to $t$ has exited $B_t$ to obtain $Z^{B_t}_t$. That is, $Z_t^{B_t}$ consists of `active' particles at time $t$; these are particles whose ancestral lines have been confined to $B_t$ up to time $t$ but may have left $B_s$ at an earlier time $s$. Similar to $Z_t$, $Z_t^{B_t}$ can be viewed as a finite discrete measure on $B_t$. Observe that the process $Z^B=(Z_t^{B_t})_{t\geq 0}$ is non-Markovian; one can see this by noticing that particles that have disappeared or been deactivated earlier may suddenly reappear or be reactivated at a later time. On the other hand, the process can be recovered from a single BBM as described above. We use $P_x$ and $E_x$, respectively, to denote the law and corresponding expectation of a BBM starting with a single particle at $x\in\mathbb{R}^d$, and by an abuse of notation, use $P_x$ and $E_x$ also for a BBM with deactivation at a boundary. For simplicity, we set $P=P_0$. 

For a Borel set $A\subseteq\mathbb{R}^d$ and $t\geq 0$, we denote by $Z_t(A)$ the mass of $Z$ inside $A$ at time $t$, and use $N_t:=Z_t(\mathbb{R}^d)$ to denote the total mass of $Z$ at time $t$. Similarly, $Z_t^{B_t}(A)$ denotes the mass of $Z_t^{B_t}$ in $A$ at time $t$. Consider a \emph{radius} function $r:\mathbb{R}_+\to \mathbb{R}_+$ which is subdiffusively increasing without bound, that is, 
\begin{equation} \label{eqradiusassumptions}
\lim_{t\rightarrow\infty}r(t)=\infty \quad\: \text{and} \quad\: r(t)=o(\sqrt{t}), \:\:\: t\rightarrow\infty . 
\end{equation}
For $t>0$, let $B_t:=B(0,r(t))$ so that $B=(B_t)_{t\geq 0}$ may be viewed as an expanding ball, and $p_t$ be the probability of confinement to $B_t$ of a standard Brownian motion (starting at the origin) over $[0,t]$. Set $n_t := Z_t^{B_t}(\mathbb{R}^d) =  Z_t^{B_t}(B_t)$ for the total mass at time $t$ of a BBM with deactivation at $\partial B$. The main objective of this work is, for a suitably decreasing function $\gamma:\mathbb{R}_+\to \mathbb{R}_+$ with $\lim_{t\rightarrow\infty} \gamma(t)=0$, to find the asymptotic behavior as $t\rightarrow\infty$ of the large-deviation probability
\begin{equation} \label{eqldprobability}
P\left(n_t<\gamma_t p_t e^{\beta t}\right) ,
\end{equation}
where we have set $\gamma_t=\gamma(t)$. It is easy via a many-to-one argument to show that 
\begin{equation} \label{eqexpectedmass}
E[n_t]=p_t e^{\beta t}.
\end{equation}
Therefore, since $\lim_{t\rightarrow\infty} \gamma_t=0$, for large $t$ one guesses that $\gamma_t p_t e^{\beta t}$ is atypically small for the mass of BBM with deactivation at $\partial B$. Theorem~\ref{theorem1} verifies that this is indeed so, and Theorem~\ref{theorem2} further shows that $p_t e^{\beta t}$ is the typical growth of mass in a certain sense. The reason why we call \eqref{eqldprobability} a large-deviation probability is explained in Remark 2. 

Branching diffusions in frozen or dynamic restricted domains in $\mathbb{R}^d$ have been widely studied over the past decades. Starting with Sevast'yanov \cite{S1958}, most of the models involved absorbing boundaries, where particles were immediately absorbed by the boundary upon hitting it. In \cite{K1978}, Kesten studied a BBM with negative drift in one dimension with absorption at the origin, starting with a particle at position $x>0$ . This model proved to be rich, leading to various fine results in subsequent works. Note that the one-dimensional model of a BBM with drift and a fixed barrier is equivalent to the case of no drift and a linearly moving barrier. More recently in \cite{H2016}, Harris et al.\ studied a BBM with drift in a fixed-size interval in $\mathbb{R}$, which is in effect a two-sided barriered version of Kesten's model. We emphasize that in all of the aforementioned works, the process studied is Markovian contrary to the non-Markovian nature of the process $Z^B$ introduced here.  

\subsection{Motivation}

The motivation to introduce and study the model of BBM with deactivation at a moving boundary arised from its intimate relation with the problem of BBM among mild Poissonian obstacles. Here, we briefly describe the connection. Let $\Pi$ be a homogeneous Poisson point process in $\mathbb{R}^d$, $(\Omega,\mathbb{P})$ be the associated probability space, and for $\omega\in\Omega$ define the \emph{trap field} with radius $a>0$ as the random set 
$$  K=K(\omega) = \bigcup_{x_i \in \text{supp}(\Pi)} \bar{B}(x_i,a), $$
where $\bar{B}(x,a)$ denotes the closed ball of radius $a$ centered at $x\in\mathbb{R}^d$. The mild obstacle rule for BBM is that when particles are inside $K$ they branch at a lower rate (possibly zero) than when they are outside $K$, where they branch at the normal rate. Hence, the random trap field serves as a mass-suppressing mechanism and in a typical environment one expects the mass of BBM to grow slower than that of a \emph{free} BBM, that is, a BBM in $\mathbb{R}^d$ without any obstacles.

In \cite{O2023}, a quenched strong law of large numbers for the mass of BBM among mild obstacles is proved (see Theorem 1 therein). It was shown that in almost every environment with respect to the Poisson point process, certain trap-free regions (called \emph{clearings}) exist, which may be suitably indexed by time $t$ in regard to the evolution of the BBM, and thus serve as expanding balls with moving boundaries, and furthermore that it is the free growth of the BBM inside these expanding clearings that determines, to the leading order, the overall growth of mass in the presence of mild obstacles. Therefore, understanding the growth of particles whose ancestral lines don't escape these clearings is essential for the proofs. We refer the reader to \cite[Section 5]{O2023} for details. 

%The current model is similarly linked to the problem of BBM among \emph{soft} obstacles, where, on top of complete suppression of branching, there is also a soft killing mechanism inside the trap field where particles are killed at a rate given by a positive bounded potential.   

\textbf{Notation:} We use $c$ as a generic positive constant, whose value may change from line to line. If we wish to emphasize the dependence of $c$ on a parameter $p$, then we write $c(p)$. We denote by $f:A\to B$ a function $f$ from a set $A$ to a set $B$. For two functions $f,g:\mathbb{R}_+\to\mathbb{R}_+$, we write $g(t)=o(f(t))$ if $g(t)/f(t)\rightarrow 0$ as $t\rightarrow\infty$. 

We denote by $X=(X_t)_{t\geq 0}$ a generic standard Brownian motion in $d$-dimensions, and use $\mathbf{P}_x$ and $\mathbf{E}_x$, respectively, as the law of $X$ started at position $x\in\mathbb{R}^d$, and the corresponding expectation. Also, we denote by $\lambda_d$ the principal Dirichlet eigenvalue of $-\frac{1}{2}\Delta$ on the unit ball in $d$ dimensions.

\textbf{Outline:} The rest of the paper is organized as follows. In Section~\ref{results}, we present our results. In Section~\ref{section3}, we develop the preparation needed for the proof of Theorem~\ref{theorem1}, which is our main result. In Section~\ref{section4} and Section~\ref{section5}, we present, respectively, the proofs of Theorem~\ref{theorem1} and Theorem~\ref{theorem2}.  

\section{Results}\label{results}

Our main result is on the large-time asymptotic behavior of the probability that the mass of BBM inside a subdiffusively expanding ball $B=(B_t)_{t\geq 0}$ with deactivation at the boundary of the ball, is atypically small. A subdiffusive expansion (see \eqref{eqradiusassumptions}) means, the ball is expanding slower than the typical rate at which a standard Brownian motion moves away from the origin, and therefore for large $t$ it would be a rare event for the Brownian motion to be confined in $B_t$. For a generic standard Brownian motion $X=(X_t)_{t\geq 0}$ and a Borel set $A\subseteq\mathbb{R}^d$, define $\sigma_A=\inf\{s\geq 0:X_s\notin A\}$ to be the first exit time of $X$ out of $A$. Denote by $a\wedge b$ the minimum of the numbers $a$ and $b$.  

\begin{theorem}[Lower large-deviations for mass of BBM in an expanding ball] \label{theorem1}
Let $r:\mathbb{R}_+ \to \mathbb{R}_+$ be subdiffusively increasing as in \eqref{eqradiusassumptions}. Let $\gamma:\mathbb{R}_+ \to \mathbb{R}_+$ be defined by $\gamma(t)=e^{-\kappa r(t)}$, where $\kappa>0$ is a constant. For $t>0$, set $B_t=B(0,r(t))$, $p_t=\mathbf{P}_0(\sigma_{B_t}\geq t)$, and $n_t=Z_t^{B_t}(B_t)$. Then, 
\begin{equation} \nonumber
\underset{t\rightarrow\infty}{\lim}\,\frac{1}{r(t)}\log P\left(n_t < \gamma_t p_t e^{\beta t}\right)= -(\kappa\wedge\sqrt{2\beta}). 
\end{equation}
\end{theorem}

\begin{remark}
Theorem~\ref{theorem1} says that there is a continuous phase transition at a critical value $\kappa=\sqrt{2\beta}$ in the asymptotic behavior of $P\left(n_t < \gamma_t p_t e^{\beta t}\right)$. This is revealed by the rate function given by $I(\kappa):=\kappa\wedge\sqrt{2\beta}$. In terms of the optimal strategies for the BBM to realize the large-deviation event $\{n_t<\gamma_t p_t e^{\beta t}\}$, the phase transition can be explained as follows.
\begin{itemize}
	\item When $\kappa\!>\!\sqrt{2\beta}$, the BBM simultaneously suppresses the branching of the initial particle and moves it outside $B(0,r(t))$ over the time interval $[0,r(t)/\sqrt{2\beta}]$. Once the initial particle is moved outside $B(0,r(t))$, the event $\{n_t=0\}$ is realized and there is no need for further atypical behavior that could incur a probabilistic cost.
	\item When $0\,{<}\,\kappa\,{\leq}\,\sqrt{2\beta}$, the BBM suppresses the branching completely over the time interval $[0,(\kappa/\beta)r(t)]$, and then behaves `normally' in the remaining interval $[(\kappa/\beta)r(t),t]$. This means, the parameter $\kappa$ is low enough so that there is no additional need to move the initial particle outside $B(0,r(t))$ over a time interval of order $r(t)$.
\end{itemize}
\end{remark}

\begin{remark}
We call $P\left(n_t < \gamma_t p_t e^{\beta t}\right)$ with $\gamma_t=e^{-\kappa r(t)}$ a large-deviation probability, because both $P(n_t=0)$ and $P\left(n_t < \gamma_t p_t e^{\beta t}\right)$ decay as $e^{-cr(t)}$, where the values $c>0$ may differ, to the leading order for large $t$. The significance of the choice $\gamma_t = e^{-\kappa r(t)}$ is as follows. It can be shown that if $\gamma_t\rightarrow 0$ as $t\rightarrow \infty$, then for all large $t$, $P(n_t<\gamma_t p_t e^{\beta t})\geq \delta \gamma_t$ for some $\delta>0$. If $\gamma_t$ decays slower than $e^{-cr(t)}$ so that $(\log\gamma_t)/r(t)\rightarrow 0$ as $t\rightarrow \infty$, this would imply $\liminf_{t\rightarrow\infty}(r(t))^{-1}\log P\left(n_t<\gamma_t p_t e^{\beta t}\right) = 0$. Therefore, in that case, in view of Theorem~\ref{theorem1}, the event $\{n_t<\gamma_t p_t e^{\beta t}\}$ would not be a large-deviation event. (See the proof of the lower bound of Theorem~\ref{theorem1} in Section~\ref{section4} for details.)
\end{remark}

\begin{theorem}[Law of large numbers for mass of BBM in an expanding ball] \label{theorem2}
Let $r:\mathbb{R}_+ \to \mathbb{R}_+$ be subdiffusively increasing as in \eqref{eqradiusassumptions}. For $t>0$, set $B_t=B(0,r(t))$, $p_t=\mathbf{P}_0(\sigma_{B_t}\geq t)$, and $n_t=Z_t^{B_t}(B_t)$. Then, 
\begin{equation} \nonumber
\underset{t\rightarrow\infty}{\lim} (r(t))^2\left(\frac{\log n_t}{t}-\beta\right) = -\lambda_d \quad \text{in $P$-probability}.
\end{equation}  
\end{theorem}

\begin{remark}
Theorem~\ref{theorem2} is called a law of large numbers, because it says that in a sense the process $n_t$ grows as its expectation as $t\to\infty$. That is, in a loose sense,
$$ \frac{\log n_t}{t}  \approx  \frac{\log E(n_t)}{t},     \quad t\to\infty.$$
\end{remark}

\section{Preparations}\label{section3}

In this section, we first list two well-known results, one concerning the distribution of mass in branching systems and the other on the hitting times of a $d$-dimensional Brownian motion. Then, we state and prove two propositions which can be obtained from existing results in a somewhat straightforward way. The results of this section will be useful in the proof of the upper bound of Theorem~\ref{theorem1}.

The following proposition is well-known from the theory of continuous-time branching processes. For a proof, see for example \cite[Section 8.11]{KT1975}. 

\begin{propa}[Distribution of mass in branching systems]\label{propa}
For a strictly dyadic continuous-time branching process $N=(N_t)_{t\geq 0}$ with constant branching rate $\beta>0$, the probability distribution at time $t$ is given by $P(N_t=k)=e^{-\beta t}(1-e^{-\beta t})^{k-1}$ for $k\geq 1$, from which it follows that
\begin{equation} P(N_t>k)=(1-e^{-\beta t})^k  \quad \text{and} \quad E[N_t]=e^{\beta t}. \nonumber 
\end{equation}
\end{propa}

As before, we use $X=(X_t)_{t\geq 0}$ to denote a generic Brownian motion in $d$ dimensions, and use $\mathbf{P}_x$ and $\mathbf{E}_x$, respectively, for the associated probability and expectation for a process that starts at position $x\in\mathbb{R}^d$. Proposition B below is on the large-time asymptotic probability of atypically large (linear) Brownian displacements. For a proof, see for example \cite[Lemma 5]{OCE2017}. 

\begin{propb}[Linear Brownian displacements]\label{propb}
For $k>0$,
\begin{equation}  \mathbf{P}_0\left(\underset{0\leq s\leq t}{\sup}|X_s|>k t\right)=\exp\left[-\frac{k^2 t}{2}(1+o(1))\right].  \nonumber 
\end{equation}
\end{propb}

As before, we use $\sigma_A=\inf\{s\geq 0: X_s\notin A \}$ as the first exit time of $X=(X_t)_{t\geq 0}$ out of A. For $a>0$, we introduce the notation $\mathbf{P}^a$ to stand for the law of a Brownian motion that starts at a distance $a$ from the origin. Also, for $a>0$, set $\tau_a=\sigma_{B(0,a)}$ for ease of notation. The following result compares the probabilities of confinement in balls for Brownian motions starting at the center of the ball and at any other point inside the ball. 

\begin{proposition}[Brownian confinement in balls, comparison] \label{prop1}
Let $a,b\in\mathbb{R}$ such that $0<a<b$. Then, there exists a positive constant $D=D(b/a,d)$ such that for all large $t$,
\begin{equation} \nonumber
D\,\mathbf{P}^a(\tau_b\geq t) \geq \mathbf{P}^0(\tau_b\geq t) .
\end{equation}
\end{proposition}

\begin{proof}
Let $\tau_b^{(\nu)}$ be the first hitting time to $b$ of the Bessel process with index $\nu$. It is well-known that if $2\nu+2$ is a positive integer, then the Bessel process is identical in law to the radial component of a $(2\nu+2)$-dimensional Brownian motion. It follows from (2.7) and (2.8) of \cite{HM2013}, respectively, that as $t\to\infty$,
\begin{equation} \label{eqhm1}
\mathbf{P}^0(\tau_b\geq t) = \frac{1}{2^{\nu-1} \Gamma(\nu+1)} e^{-\frac{j_{\nu,1}^2 t}{2b^2}}(1+o(1))
\end{equation}
and
\begin{equation} \label{eqhm2}
\mathbf{P}^a(\tau_b\geq t) = 2\left(\frac{b}{a}\right)^\nu \frac{J_\nu(a j_{\nu,1}/b)}{j_{\nu,1}J_{\nu+1}(j_{\nu,1})} e^{-\frac{j_{\nu,1}^2 t}{2b^2}}(1+o(1))   ,
\end{equation}
where $0<a<b$, $J_\mu$ is the Bessel function of the first kind of order $\mu$, $\{j_{\mu,k}\}_{k=1}^\infty$ is the increasing sequence of positive zeros of $J_\mu$, and we suppress the dependence of $\tau$ on $\nu$ (hence on $d$) in notation. It then follows from \eqref{eqhm1} and \eqref{eqhm2} that there exist $t_0>0$ and a positive constant $D=D(b/a,d)$ such that for all $t\geq t_0$,
$$ D\,\mathbf{P}^a(\tau_b\geq t) \geq \mathbf{P}^0(\tau_b\geq t) .  $$  
\end{proof}

For $0<\delta<1$, $k\geq 0$, and $B_t=B(0,r(t))$, let 
\begin{equation}
\widetilde{p}_t=\mathbf{P}^{\delta r(t)}\left(\sigma_{B_t}\geq t-kr(t)\right)  \label{eqptilde} .
\end{equation}
Observe that by Brownian scaling, we have $\widetilde{p}_t = \mathbf{P}^{\delta}\left(\sigma_{B(0,1)}\geq (t-kr(t))/(r(t))^2\right)$. The following result is on the aytpically small growth of mass of a BBM with deactivation at the boundary of a subdiffusively expanding ball, where the BBM is started with a single particle at an interior point of the ball whose distance to the center is on the scale of the radius of the ball. For $a>0$, denote by $P^a$ the law of a branching Brownian motion that starts with a single particle at distance $a$ from the origin.  

\begin{proposition} \label{prop2}
Let $0\,{<}\,\delta\,{<}\,1$, $0\,{\leq}\,k_1\,{<}\,k_2$ and $r:\mathbb{R}_+ \to \mathbb{R}_+$ be subdiffusively increasing as in \eqref{eqradiusassumptions}. Let $\gamma_t=e^{-\kappa r(t)}$, where $\kappa>0$ is a constant, and $\widetilde{p}_t=\widetilde{p}_t(k)$ be as in \eqref{eqptilde}. For $t>0$, set $B_t=B(0,r(t))$ and $n_t=Z_t^{B_t}(B_t)$. Then, there exists a constant $c=c(\delta,\kappa,\beta)>0$ such that for all large $t$,
\begin{equation} \label{eqprop2}
\sup_{k_1\leq k\leq k_2}P^{\delta r(t)}\left(n_{t-kr(t)} < \gamma_t \widetilde{p}_t e^{\beta(t-kr(t))} \right) \leq e^{-c r(t)} .
\end{equation}
\end{proposition}

The following proof follows closely the proof of the upper bound of \cite[Theorem 2]{O2023}.

\begin{proof}
Let $k\in[k_1,k_2]$ and set $g_t=2\gamma_t$. Recall that $N_t=Z_t(\mathbb{R}^d)$, where $Z=(Z_t)_{t\geq 0}$ denotes a BBM. For $t>0$, start with the estimate
\begin{equation} \label{prop2eq1}
P^{\delta r(t)}(\:\cdot\:) \leq P^{\delta r(t)}\left(\:\cdot\: \big|\: N_{t-kr(t)} > e^{\beta(t-kr(t))} g_t \right) + P^{\delta r(t)}\left( N_{t-kr(t)} \leq e^{\beta(t-kr(t))} g_t\right) .
\end{equation}
Proposition A yields $P(N_{t-kr(t)}\leq n)=1-(1-e^{-\beta(t-kr(t))})^n\leq n e^{-\beta(t-kr(t))}$ for any $n\geq 1$. Setting $n=\left\lfloor e^{\beta(t-kr(t))} g_t\right\rfloor$, we have for $t>0$, 
\begin{equation} \label{yenidenklem}
P\left(N_{t-kr(t)} \leq e^{\beta(t-kr(t))} g_t\right) = P\left(N_{t-kr(t)} \leq \left\lfloor e^{\beta(t-kr(t))} g_t\right\rfloor\right) \leq g_t ,
\end{equation}
which bounds the second term on the right-hand side of \eqref{prop2eq1}. Since $k_2>0$ is fixed, and $r(t)=o(\sqrt{t})$ by assumption, it is clear that there exists $t_0>0$ such that $t-k_2r(t)>0$ and $\left\lfloor e^{\beta(t-k_2 r(t))} g_t\right\rfloor\geq 1$ for all $t\geq t_0$. Fix this $t_0$, and for $t\geq t_0$, define
$$  \widetilde{P}_t\left(\:\cdot\:\right) = P^{\delta r(t)}\left(\:\cdot\: \big|\: N_{t-kr(t)} > e^{\beta(t-kr(t))} g_t\right)  , $$
and let $\widetilde{E}_t$ and $\widetilde{\text{Var}}_t$ denote, respectively, the expectation and variance associated to $\widetilde{P}_t$. Let $\mathcal{N}_t$ denote the set of particles of $Z$ at time $t$, and for $u\in\mathcal{N}_t$, let $(Y_u(s))_{0\leq s\leq t}$ denote the ancestral line up to $t$ of particle $u$. Now, conditional on the event $\{N_{t-kr(t)} > e^{\beta(t-kr(t))} g_t\}$, choose randomly, independent of their genealogy and position, $M_t:=\left\lfloor e^{\beta(t-kr(t))} g_t\right\rfloor$ particles out of the particles at time $t$. Denote the collection of the chosen particles by $\mathcal{M}_t$, and define
$$ \hat{n}_t = \sum_{u\in\mathcal{M}_t} \mathbb{1}_{A_u} , $$
where $A_u = \{Y_u(s) \in B_t \:\:\forall\: 0\leq s\leq t-kr(t) \}$, and we have suppressed the dependence of $A_u$ on $t$ in notation. Since the collection $\mathcal{M}_t$ is chosen independently of the motion process, the ancestral line of each particle in $\mathcal{M}_t$ is Brownian and the many-to-one lemma implies that
\begin{equation} \nonumber
\widetilde{E}_t\left[\hat{n}_t\right] = \widetilde{p}_t \left\lfloor e^{\beta(t-kr(t))} g_t\right\rfloor ,
\end{equation} 
where $\widetilde{p}_t$ is as in \eqref{eqptilde}. Now apply Chebyshev's inequality to the random variable $\hat{n}_t$ to obtain, by \eqref{prop2eq1} and \eqref{yenidenklem}, for $t\geq t_0$,
\begin{align}
P^{\delta r(t)}\left(n_{t-kr(t)} < \gamma_t \widetilde{p}_t e^{\beta(t-kr(t))} \right) &\leq \widetilde{P}_t\left(\hat{n}_t < \gamma_t \widetilde{p}_t e^{\beta(t-kr(t))}\right) + g_t \nonumber \\
& \leq \widetilde{P}_t\left(|\hat{n}_t-\widetilde{E}_t[\hat{n}_t]|> \widetilde{p}_t\left\lfloor e^{\beta(t-kr(t))} g_t\right\rfloor-\gamma_t \widetilde{p}_t e^{\beta(t-kr(t))}\right) + g_t  \nonumber \\
& \leq \frac{\widetilde{\text{Var}}_t(\hat{n}_t)}{[(g_t-\gamma_t) \widetilde{p}_t e^{\beta(t-kr(t))}-\widetilde{p}_t]^2} + g_t . \label{prop2eq3}
\end{align}

Similar to the proof of Theorem 2 in \cite{O2023}, we may bound $\widetilde{\text{Var}}_t(\hat{n}_t)$ from above as
\begin{equation} \label{prop2eq4}
\widetilde{\text{Var}}_t\left(\hat{n}_t\right) \leq g_t e^{\beta(t-kr(t))}\left(\widetilde{p}_t-\widetilde{p}_t^2\right) + g_t^2 e^{2\beta(t-kr(t))}\left[\left(\mathcal{E}\otimes\widetilde{P}_t\right)(A_i\cap A_j)-\widetilde{p}_t^2\right],
\end{equation} 
where $(\mathcal{E}\otimes \widetilde{P}_t)(A_i\cap A_j)$ denotes averaging $\widetilde{P}_t(A_i\cap A_j)$ over the $M_t(M_t-1)$ possible pairs in the randomly chosen set $\mathcal{M}_t$. Define
\begin{equation} \label{eqivirzivir}
\widetilde{p}^{(t)}(x,s,dy):=\mathbf{P}_x(X_s\in dy \mid X_z\in B_t \:\:\forall\, 0\leq z\leq s) \quad \text{and} \quad p_{s,x}^t:=\mathbf{P}_x(\sigma_{B_t}\geq s)  . 
\end{equation}
Let $Q^{(t)}$ be the distribution of the splitting time of the most recent common ancestor of $i$th and $j$th particles under $\mathcal{E}\otimes\widetilde{P}_t$. Then, applying the Markov property at this splitting time, we obtain
\begin{equation} \label{prop2eq5}
\left(\mathcal{E}\otimes\widetilde{P}_t\right)(A_i\cap A_j) = \widetilde{p}_t \int_0^{t-kr(t)} \int_{B_t} p^t_{t-kr(t)-s,x}\, \widetilde{p}^{(t)}(\delta r(t)\mathbf{e},s,dx) \,Q^{(t)}(ds) ,
\end{equation}
where $\mathbf{e}$ denotes any unit vector in $\mathbb{R}^d$. Note that an application of the Markov property of a standard Brownian motion at time $s$ with $0<s<t$ yields
\begin{equation} \label{prop2eq6}
\widetilde{p}_t = p^t_{s,\delta r(t) \mathbf{e}} \int_{B_t} p^t_{t-kr(t)-s,y} \,\widetilde{p}^{(t)}(\delta r(t)\mathbf{e},s,dy). 
\end{equation}
Set $p^t_{s,\delta r(t)}=p^t_{s,\delta r(t)\mathbf{e}}$ for simplicity. It then follows from \eqref{prop2eq5} and \eqref{prop2eq6} that
\begin{equation} \label{prop2eq7}
\left(\mathcal{E}\otimes\widetilde{P}_t\right)(A_i\cap A_j) = \widetilde{p}_t^2 \int_0^{t-kr(t)} \frac{1}{p_{s,\delta r(t)}^t}\, Q^{(t)}(ds) .
\end{equation}
For $t>0$, define
$$ J_t:=\int_0^{t-kr(t)} \frac{1}{p_{s,\delta r(t)}^t}\, Q^{(t)}(ds) . $$
Then, by \eqref{prop2eq4} and \eqref{prop2eq7}, we have for all $t\geq t_0$, 
\begin{equation} \label{prop2eq8}
\widetilde{\text{Var}}_t(\hat{n}_t)\leq g_t \widetilde{p}_t e^{\beta (t-kr(t))}+g_t^2 \widetilde{p}_t^2 e^{2\beta (t-kr(t))}(J_t-1).
\end{equation}
Observe that $J_t-1\geq 0$. Next, we bound $J_t-1$ from above.

Choose $c>0$, fix it, and for $t$ large enough so that $t-k_2 r(t)>c r(t)$, define 
$$ J_t^{(1)} = \int_0^{c r(t)} \frac{1}{p_{s,\delta r(t)}^t}\, Q^{(t)}(ds), \quad \quad J_t^{(2)} = \int_{c r(t)}^{t-kr(t)} \frac{1}{p_{s,\delta r(t)}^t}\, Q^{(t)}(ds)   . $$
Split $J_t$ as $J_t=J_t^{(1)}+J_t^{(2)}$. In what follows, to bound $J_t$ from above for large $t$, we will use that over the first time interval $[0,cr(t)]$, the integrand $1/p_{s,\delta r(t)}^t$ is small enough; whereas the distribution of $Q^{(t)}$ puts a small enough `weight' on the second time interval $[cr(t),t-kr(t)]$. 

Since $p_{s,\delta r(t)}^t$ is nonincreasing in $s$, $J_t^{(1)} \leq \left[p_{c r(t) ,\delta r(t)}^t\right]^{-1}$. Moreover, by Proposition B, 
$$ 1-p_{c r(t) ,\delta r(t)}^t = \exp\left[-\frac{(1-\delta)^2 r(t)}{2c}(1+o(1))\right]  , $$
from which it follows that
\begin{equation} \label{prop2eq9}
J_t^{(1)} -1 \leq \exp\left[-\frac{(1-\delta)^2 r(t)}{2c}(1+o(1))\right] .
\end{equation}

Next, we bound $J_t^{(2)}$ from above. It is known from \cite[Proposition 5]{E2008} that $Q^{(t)}$ is absolutely continuous with respect to the Lebesgue measure, which we denote by $ds$, and its density function, which we denote by $g^{(t)}$, satisfies 
\begin{equation} \nonumber
\exists\,C>0,\: s_0>0 \quad \text{such that} \quad \forall\:s\geq s_0,\:\: g^{(t)}(s)\leq C s e^{-\beta s}.
\end{equation}
Recall that $\lambda_d$ denotes the principal Dirichlet eigenvalue of $-\frac{1}{2}\Delta$ on the unit ball in $d$ dimensions. It follows from \cite[(2.8)]{HM2013} (since $2\lambda_d = j_{\nu,1}^2$ therein) and Brownian scaling that there exists a positive constant $D=D(\delta,d)$ such that for all $t>0$,
\begin{equation} \nonumber
p^t_{s,\delta r(t)} \geq D e^{-\frac{\lambda_d s}{(r(t))^2}} .
\end{equation}
We may then continue with
\begin{equation} \label{prop2eq11}
J_t^{(2)} = \int_{c r(t)}^{t-kr(t)} \frac{1}{p_{s,\delta r(t)}^t}\, Q^{(t)}(ds) \leq \frac{C}{D} \int_{c r(t)} ^{\infty} s \exp\left[-\left(\beta-\frac{\lambda_d}{(r(t))^2}\right)s\right]\,ds \leq e^{-\beta c r(t)(1+o(1))},
\end{equation}
where the last step follows by an application of integration by parts along with the assumptions that $r(t)\to\infty$ as $t\to\infty$ and $r(t)=o(\sqrt{t})$. It follows from \eqref{prop2eq9} and \eqref{prop2eq11} that there exists a positive constant $c(\delta,\beta)$ such that for all large $t$, 
\begin{equation} \label{prop2eq12}
J_t-1 = J_t^{(1)} -1 + J_t^{(2)} \leq e^{-c(\delta,\beta) r(t)} . 
\end{equation}
The bound in \eqref{prop2eq12} has no dependence on $k$. Moreover, the bounds in \eqref{prop2eq3} and \eqref{prop2eq8} hold for each $k\in [k_1,k_2]$. Then, from \eqref{prop2eq3}, \eqref{prop2eq8} and \eqref{prop2eq12}, we have
\begin{equation} \nonumber
P^{\delta r(t)}\left(n_{t-kr(t)} < \gamma_t \widetilde{p}_t e^{\beta(t-kr(t))} \right) \leq \frac{g_t \widetilde{p}_t e^{\beta (t-kr(t))}+g_t^2 \widetilde{p}_t^2 e^{2\beta (t-kr(t))}e^{-c(\delta,\beta) r(t)}}{[(g_t-\gamma_t) \widetilde{p}_t e^{\beta(t-kr(t))}-\widetilde{p}_t]^2} + g_t.
\end{equation}
To complete the proof, note that $\widetilde{p}_t\to 0$ as $t\to\infty$, and recall the choice $g_t=2\gamma_t$ and that $\gamma_t=e^{-\kappa r(t)}$ for some $\kappa>0$. This yields, there exists $t_1>0$ such that for all $t\geq t_1$ and $k\in[k_1,k_2]$,
$$  P^{\delta r(t)}\left(n_{t-kr(t)} < \gamma_t \widetilde{p}_t e^{\beta(t-kr(t))} \right) \leq \frac{4}{\gamma_t \widetilde{p}_t} e^{-\beta(t-kr(t))} + 8 e^{-c(\delta,\beta) r(t)} + g_t \leq 11 e^{-(c(\delta,\beta)\wedge \kappa)r(t)} . $$

\end{proof}

\section{Proof of Theorem~\ref{theorem1}}\label{section4}

\subsection{Proof of the lower bound}

The proof of the lower bound of Theorem~\ref{theorem1} is based on finding an optimal strategy to realize the event $\{n_t<\gamma_t p_t e^{\beta t}\}$ for each of the low $\kappa$ regime $0<\kappa\leq \sqrt{2\beta}$ and the high $\kappa$ regime $\kappa>\sqrt{2\beta}$, and it was given in \cite{O2023}. Here, we review the proof for completeness.

Note that $\{n_t=0\}\subseteq\{n_t<\gamma_t p_t e^{\beta t}\}$, and one way to realize the event $\{n_t=0\}$ is to completely suppress the branching of the initial particle and move it outside $B_t=B(0,r(t))$ over the time interval $[0, kr(t)]$ for some $k>0$. By Proposition B, the probability of this joint strategy is
\begin{equation} \label{eq1lowerbound}
\exp\left[-\beta k r(t)-\frac{r(t)}{2 k}(1+o(1))\right].
\end{equation}
Optimizing the exponent in \eqref{eq1lowerbound} over $k>0$ gives $k=1/\sqrt{2\beta}$, and with this choice of $k$, we obtain
\begin{equation} \label{eq2lowerbound}
P(n_t<\gamma_t p_t e^{\beta t})\geq P(n_t=0)\geq \exp\left[-\sqrt{2\beta}r(t)(1+o(1))\right] .
\end{equation}

Now let $f:\mathbb{R}_+\to\mathbb{R}_+$ be such that $f(t)=o(t)$, and denote by $\tau_1$ and $(Y_1(s))_{0\leq s\leq \tau_1}$, respectively, the lifetime and the path of the initial particle. For $t>0$, define the events
$$ A_t=\{N_{f(t)}=1\}, \quad E_t=\{n_t<\gamma_t p_t e^{\beta t} \}, \quad D_t=\{ Y_1(z)\in B_t\:\:\forall\, 0\leq z \leq f(t) \} .$$
Estimate
\begin{equation} \label{eq3lowerbound}
P(E_t)\geq P(E_t \cap A_t)=P(E_t \mid A_t) P(A_t).
\end{equation}
Conditional on $A_t$, it is clear that $\tau_1\geq f(t)$, and $n_t=0$ if $Y_1(z)\notin B_t$ for some $z\in[0,f(t)]$. Hence,
\begin{equation} \label{eq4lowerbound}
E\left[n_t \mid A_t\right]= E\left[n_t \mathbb{1}_{D_t} \mid A_t\right] = E\left[n_t \mid A_t, D_t\right] P(D_t \mid A_t).
\end{equation}
Using the notation from \eqref{eqivirzivir}, we have
\begin{align} 
E\left[n_t \mid A_t, D_t\right] &= \int_{B_t} E\left[n_t \mid A_t, D_t, Y_1(f(t))=y\right]P(Y_1(f(t)) \in dy \mid A_t,D_t) \nonumber \\
&= \int_{B_t} E\left[n_t \mid A_t, D_t, Y_1(f(t))=y\right] \widetilde{p}^{(t)}(0,f(t),dy) . \label{eq5lowerbound}
\end{align}
and 
\begin{equation} \label{eq6lowerbound}
P(D_t \mid A_t) = p_{f(t),0}^t .
\end{equation}
Moreover, applying the Markov property of the BBM at time $f(t)$, and using the many-to-one lemma (see for instance \cite[Lemma 1.6]{E2014}), we obtain
\begin{equation} \label{eq7lowerbound}
E\left[n_t \mid A_t, D_t, Y_1(f(t))=y\right] = p^t_{t-f(t),y}\, e^{\beta(t-f(t))}  ,\quad y\in B_t . 
\end{equation}
It then follows from \eqref{eq4lowerbound}-\eqref{eq7lowerbound} that
\begin{equation} \nonumber
E\left[n_t \mid A_t\right] = e^{\beta(t-f(t))} p_{f(t),0}^t \int_{B_t} p^t_{t-f(t),y}\, \widetilde{p}^{(t)}(0,f(t),dy) =  e^{\beta(t-f(t))} p_t ,
\end{equation}
where the last equality follows by applying the Markov property of Brownian motion at time $f(t)$, similar in spirit to \eqref{prop2eq6}. Then, by the Markov inequality,   
\begin{equation} \label{eq9lowerbound}
P(E_t^c \mid A_t)\leq \frac{E\left[n_t \big\vert A_t\right]}{\gamma_t p_t e^{\beta t}}=\gamma_t^{-1}e^{-\beta f(t)}.
\end{equation}
Choose $f(t)=-(1/\beta)\log ((1-\delta)\gamma_t)$ with $0<\delta<1$ in \eqref{eq9lowerbound}, which leads to $P(E_t \mid A_t)\geq \delta$. Noting that $P(A_t)=e^{-\beta f(t)}$, the estimate in \eqref{eq3lowerbound} then yields
\begin{equation} \nonumber
P(E_t)\geq \delta e^{-\beta f(t)} = \delta(1-\delta)\gamma_t = e^{-\kappa r(t)(1+o(1))}.
\end{equation}
In view of \eqref{eq2lowerbound}, this completes the proof of the lower bound of Theorem~\ref{theorem1}.

\subsection{Proof of the upper bound}

We will follow a discretization method similar to the proof of the upper bound of Theorem 2.1 in \cite{O2020}. Proposition~\ref{prop1} and Proposition~\ref{prop2} will be the key ingredients in the proof.

Recall that $N_t=Z_t(\mathbb{R}^d)$ is the total mass at time $t$, and define the random variable
$$ \rho_t = \sup\{\rho\in [0,1]:N_{\rho(\kappa/\beta)r(t)}\leq \left\lfloor r(t)\right\rfloor\}  .$$
Observe that for $x\in [0,1]$, we have $\{\rho_t\geq x\}\subseteq \{N_{x(\kappa/\beta)r(t)}\leq \left\lfloor r(t)\right\rfloor+1\}$. Recall that $\gamma_t=e^{-\kappa r(t)}$ and $p_t=\mathbf{P}_0(\sigma_{B_t}\geq t)$. For $t>0$, define the event
$$ A_t=\{n_t<\gamma_t p_t e^{\beta t}  \} . $$
We condition on $\rho_t$ as follows. For every $n=2,3,\ldots$ 
\begin{align}
P(A_t)&=\sum_{i=0}^{n-2} P\left(A_t\cap\left\{\frac{i}{n}\leq \rho_t<\frac{i+1}{n}\right\}\right) + P\left(A_t\cap\left\{\rho_t\geq 1-\frac{1}{n}\right\}\right) \nonumber \\
&\leq \sum_{i=0}^{n-2} \exp\left[-\frac{i}{n}\kappa r(t)+o(r(t))\right] P_t^{(i,n)}(A_t) +\exp\left[-\kappa r(t)\left(1-\frac{1}{n}\right)+o(r(t))\right] \label{eqt1} ,
\end{align}
where we have used Proposition A to bound $P(\frac{i}{n}\leq \rho_t<\frac{i+1}{n})$ and $P(\rho_t\geq 1-1/n)$ from above, and introduced the conditional probabilities
$$  P_t^{(i,n)}(\:\cdot\:) = P\left(\:\cdot\: \bigg| \: \frac{i}{n}\leq \rho_t<\frac{i+1}{n} \right)   , \quad i=0,1,\ldots,n-2. $$
Next, we bound $P_t^{(i,n)}(A_t)$ from above. To that end, let $E_t$ be the event that there is at least one particle outside the ball $B(0,\delta r(t))$ at some instant $s$ with $0\leq s\leq \rho_t(\kappa/\beta)r(t)$, and continue with the estimate
\begin{equation}
P_t^{(i,n)}(A_t) \leq P_t^{(i,n)}\left(E_t\right) + P_t^{(i,n)}\left(A_t \mid E_t^c\right) \label{eqt2} .
\end{equation}
Under the law $P_t^{(i,n)}$, we have $\rho_t<(i+1)/n$, and by definition of $\rho_t$ there are exactly $\left\lfloor r(t)\right\rfloor+1$ particles present at time $\rho_t(\kappa/\beta)r(t)$. Therefore, the first term on the right-hand side of \eqref{eqt2} can be estimated via Proposition B and the union bound as
\begin{align} 
&P_t^{(i,n)}\left(E_t\right) \leq (\left\lfloor r(t)\right\rfloor+1) \mathbf{P}_0 \bigg(\sup_{0\leq s\leq \frac{i+1}{n}\frac{\kappa}{\beta}r(t)} |X_s|>\delta r(t)\bigg) \nonumber \\
&=\exp\left[-\frac{1}{2}\left(\frac{\delta\beta n}{(i+1)\kappa}\right)^2\frac{i+1}{n}\frac{\kappa}{\beta}r(t)(1+o(1))\right] =\exp\left[-\frac{\delta^2 n \beta r(t)}{2(i+1)\kappa}(1+o(1))\right]  . \label{eqt3}
\end{align}
We now bound the term $P_t^{(i,n)}\left(A_t \mid E_t^c\right)$ in \eqref{eqt2} from above. Observe that conditional on $\{\frac{i}{n}\leq \rho_t<\frac{i+1}{n}\}\cap E_t^c$, there exists an instant, namely $\rho_t(\kappa/\beta)r(t)$, inside the time interval $[i/n(\kappa/\beta)r(t),(i+1)/n(\kappa/\beta)r(t)]$, at which there are exactly $\left\lfloor r(t)\right\rfloor+1$ particles, all of which are inside $B(0,\delta r(t))$ and have ancestral lines confined to $B(0,\delta r(t))$ over $[0,\rho_t(\kappa/\beta)r(t)]$. For an upper bound on $P_t^{(i,n)}\left(A_t \mid E_t^c\right)$, we suppose the `worst case', that is, suppose that each of these particles is on the boundary of $B(0,\delta r(t))$ at time $\rho_t(\kappa/\beta)r(t)$. Then, an application of the strong Markov property of the BBM at time $\rho_t(\kappa/\beta)r(t)$ along with the independence of the sub-BBMs initiated at that time yields
\begin{equation} \label{eqt4}
P_t^{(i,n)}\left(A_t \mid E_t^c\right) \leq \left(\sup_{\frac{i}{n}\leq x\leq \frac{i+1}{n}}\left[P^{\delta r(t)}\left(n_{t-x(\kappa/\beta)r(t)}<\gamma_t p_t e^{\beta t}\right)\right]\right)^{\left\lfloor r(t)\right\rfloor} .
\end{equation}

%Argument begins here.

Next, we seek to find a suitable upper bound for $P^{\delta r(t)}\left(n_{t-x(\kappa/\beta)r(t)}<\gamma_t p_t e^{\beta t}\right)$ uniformly over $x\in[i/n,(i+1)/n]$. Let $0\leq k_1<k_2$. Provided $\kappa-k_2\beta>0$, there exists $t_0>0$ such that for all $t\geq t_0$ and $k\in[k_1,k_2]$, the following bound on $\gamma_t p_t e^{\beta t}$ holds:
\begin{align} 
\gamma_t p_t e^{\beta t} &= e^{-\kappa r(t)}\mathbf{P}_0\left(\sigma_{B_t}\geq t\right) e^{\beta t} \nonumber \\
&= e^{-(\kappa-k\beta)r(t)}\mathbf{P}_0\left(\sigma_{B_t}\geq t\right) e^{\beta (t-k r(t))} \nonumber \\
&\leq D\,\xbar{\gamma}_t\,\mathbf{P}^{\delta r(t)}\left(\sigma_{B_t}\geq t\right) e^{\beta (t-k r(t))} \nonumber \\
&\leq D\,\xbar{\gamma}_t\, \mathbf{P}^{\delta r(t)}\left(\sigma_{B_t}\geq t-k r(t)\right) e^{\beta (t-k r(t))} \nonumber \\ 
&\leq e^{-(\kappa-k_2\beta)r(t)/2} \widetilde{p}_t e^{\beta (t-k r(t))} , \label{eqt41}
\end{align}
where $D=D(\delta,d)>0$, we have defined $\xbar{\gamma}_t:=e^{-(\kappa-k\beta)r(t)}$ and used Proposition~\ref{prop1} together with Brownian scaling in the first inequality, and used the monotonicity in $s$ of $\mathbf{P}^{\,\cdot}\left(\sigma_{B_t}\geq s\right)$ in the second inequality. Now set $k_1=(i/n)(\kappa/\beta)$ and $k_2=((i+1)/n)(\kappa/\beta)$. Observe that $\kappa-k_2\beta=\kappa(1-(i+1)/n)>0$ since $i\leq n-2$. With these choices of $k_1$ and $k_2$, it then follows from \eqref{eqt41} and Proposition~\ref{prop2} upon replacing $\gamma_t$ by $\widetilde{\gamma}_t:=e^{-(\kappa-k_2\beta)r(t)/2}$ in \eqref{eqprop2} and setting $x=k\beta/\kappa$ that there exist $c=c(\delta,\kappa,\beta)>0$ and $t_1>0$ such that for all $x\in[i/n,(i+1)/n]$ and $t\geq t_1$,
\begin{equation} \label{eqt5}
P^{\delta r(t)}\left(n_{t-x(\kappa/\beta)r(t)}<\gamma_t p_t e^{\beta t}\right) \leq P^{\delta r(t)}\left(n_{t-x(\kappa/\beta)r(t)}<\widetilde{\gamma}_t \widetilde{p}_t e^{\beta(t-x(\kappa/\beta) r(t))}\right) \leq e^{-c r(t)} .
\end{equation}

%Argument ends here.

Combining \eqref{eqt2}-\eqref{eqt4} and \eqref{eqt5}, for all large $t$,
\begin{equation} \label{eqt6}
P_t^{(i,n)}(A_t) \leq \exp\left[-\frac{\delta^2 n \beta r(t)}{2(i+1)\kappa}(1+o(1))\right] + \left(e^{-c r(t)}\right)^{\left\lfloor r(t)\right\rfloor}.
\end{equation} 
Observe that the first term on the right-hand side of \eqref{eqt6} is the dominating term for large $t$. Substituting \eqref{eqt6} into \eqref{eqt1}, and optimizing over $i\in\{0,1,\ldots,n-2\}$ yields
\begin{equation}  \label{eqt7}
\limsup_{t\to\infty}\frac{1}{r(t)}\log P(A_t)\leq -\beta\left[\min_{i\in\{0,1,\ldots,n-2\}}\left\{\frac{i}{n}\frac{\kappa}{\beta}+\frac{\delta^2}{2\kappa(i+1)/n}\right\}\wedge \frac{\kappa}{\beta}\left(1-\frac{1}{n}\right)\right] .
\end{equation}
First, let $\delta\to 1$ on the right-hand side of \eqref{eqt7}; then set $\rho=i/n$ to obtain
\begin{equation}  \nonumber
\limsup_{t\to\infty}\frac{1}{r(t)}\log P(A_t)\leq -\beta\left[\min_{\rho\in\{0,1/n,\ldots,(n-2)/n\}}\left\{\frac{\rho\kappa}{\beta}+\frac{1}{2\kappa\rho+2\kappa/n}\right\}\wedge \frac{\kappa}{\beta}\left(1-\frac{1}{n}\right)\right] .
\end{equation}
Now let $n\to\infty$ and use the continuity of the functional form from which the minimum is taken to obtain
\begin{equation}  \label{eqt9}
\limsup_{t\to\infty}\frac{1}{r(t)}\log P(A_t)\leq -\beta\left[\inf_{\rho\in(0,1]}\left\{\frac{\rho\kappa}{\beta}+\frac{1}{2\kappa\rho}\right\}\wedge \frac{\kappa}{\beta}\right].
\end{equation}
For $\rho\in(0,1]$, define the function $f$ by
$$ f(\rho)= \frac{\rho\kappa}{\beta}+\frac{1}{2\kappa\rho} . $$
One can check that when $\kappa>\sqrt{2\beta}$, $f$ is minimized at $\bar{\rho}=\sqrt{\beta/(2\kappa^2)}$ over $(0,1]$ where $0<\bar{\rho}<1/2$, and the minimum value of $f$ is $\sqrt{2/\beta}$. On the other hand, when $\kappa\leq\sqrt{2\beta}$, the minimum value of $f$ over $(0,\infty)$ is $\sqrt{2/\beta}$, which means, due to $\kappa/\beta \leq \sqrt{2/\beta}$, the second term under the minimum in \eqref{eqt9} is the output of this minimum. Collecting all this regarding $f$, and using \eqref{eqt9}, we arrive at
$$  \limsup_{t\to\infty}\frac{1}{r(t)}\log P(A_t)\leq -(\kappa \wedge \sqrt{2\beta})  . $$
This completes the proof of the upper bound of Theorem~\ref{theorem1}.

\section{Proof of Theorem~\ref{theorem2}}\label{section5}

\subsection{Proof of the upper bound}

We use the Markov inequality together with the known formula for $E(n_t)$. Recall from \eqref{eqexpectedmass} that $E(n_t)=p_t e^{\beta t}$, where $p_t$ stands for the probability of confinement of a Brownian motion to $B(0,r(t))$ over $[0,t]$. It follows from a well-known result on Brownian confinement in balls along with Brownian scaling that $p_t=\mathbf{P}_0\left(\sigma_{B_t}\geq t\right)=\exp\left[-\lambda_d t(1+o(1))/(r(t))^2\right]$ (see for example \cite[Proposition B]{O2023}). Then, by the Markov inequality, for any $\varepsilon>0$,
\begin{align}
P\left(n_t>\exp\left[\beta t-\frac{(\lambda_d-\varepsilon)t}{(r(t))^2}\right]\right) &\leq \frac{E(n_t)}{\exp\left[\beta t-\frac{(\lambda_d-\varepsilon)t}{(r(t))^2}\right]} = \frac{\exp\left[-\frac{\lambda_d t}{(r(t))^2}(1+o(1))+\beta t\right]}{\exp\left[\beta t-\frac{(\lambda_d-\varepsilon)t}{(r(t))^2}\right]} \nonumber \\
& = \exp\left[-\frac{t}{(r(t))^2}(\varepsilon+o(1))\right] \to 0, \quad t\to\infty . \nonumber
\end{align}
This proves the upper bound of Theorem~\ref{theorem2}.

\subsection{Proof of the lower bound}

To prove the lower bound of Theorem~\ref{theorem2}, we will show that for any $\varepsilon>0$, as $t\to\infty$, $P\left(n_t<\exp\left[\beta t-(\lambda_d+\varepsilon)t/(r(t))^2\right]\right)\to 0$. We consider two cases: $r(t)\lesssim t^{1/3}$ and $r(t)\gtrsim t^{1/3}$, where we use the notation $f(t)\gtrsim g(t)$ to mean there exists a positive constant $c>0$ such that $f(t)/g(t)\geq c$ for all $t\geq t_0$ for some $t_0>0$.

\noindent \textbf{\underline{Case 1}:} Suppose that $\lim_{t\to\infty}r(t)=\infty$ and $r(t)\lesssim t^{1/3}$. Then, $e^{-\varepsilon t/(r(t))^2}\leq e^{-\kappa r(t)}$ for some $\kappa=\kappa(\varepsilon)>0$ for all large $t$, and it follows that for all large $t$,
\begin{equation} \nonumber
P\left(n_t<\exp\left[\beta t-\frac{(\lambda_d+\varepsilon)t}{(r(t))^2}\right]\right) \leq P\left(n_t<\exp\left[\beta t-\frac{\lambda_d t}{(r(t))^2}-\kappa r(t)\right]\right) \to 0,\quad t\to\infty, 
\end{equation}
where we have used Theorem~\ref{theorem1} and that $p_t=\exp\left[-\frac{\lambda_d t}{(r(t))^2}(1+o(1))\right]$.

\noindent \textbf{\underline{Case 2}:} Suppose that $r(t)=o(\sqrt{t})$ and $r(t)\gtrsim t^{1/3}$. To handle this case, we revisit the proof of the upper bound of Theorem 2 in \cite{O2023}. Observe that the aforementioned proof remains the same up to and including equation (92) therein provided that $g_t>\gamma_t$ for all $t>0$ and that $\lim_{t\to\infty}g_t = 0$. Now, choose for instance $g_t=2 \gamma_t$ and $\gamma_t = e^{-\frac{\varepsilon t}{2(r(t))^2}}$. Then, by \cite[(92)]{O2023}, for all large $t$, 
\begin{align}
P\left(n_t<\exp\left[\beta t-\frac{(\lambda_d+\varepsilon)t}{(r(t))^2}\right]\right) &\leq P\left(n_t<\gamma_t p_t e^{\beta t}\right) \nonumber \\
&\leq \frac{4}{\gamma_t p_t}e^{-\beta t} + 8e^{-\sqrt{\beta/2}r(t)} + g_t \leq 3 g_t. \nonumber
\end{align}
Since $g_t\to 0$ as $t\to\infty$, this completes the proof of the lower bound of Theorem~\ref{theorem2}.

%\section*{References}

\bibliographystyle{plain}

\begin{thebibliography}{1}

\bibitem{E2014}
J. Engl\"{a}nder.
\newblock {\em Spatial Branching in Random Environments and with Interaction}. World Scientific, Singapore, 2014.

\bibitem{E2008}
J. Engl\"{a}nder.
\newblock Quenched law of large numbers for branching Brownian motion in a random medium.
\newblock{\em Ann.\ Inst.\ H.\ Poincar\'e Probab.\ Statis.} \textbf{44} (3) (2008) 490 -- 518.

\bibitem{HM2013}
Y. Hamana and H. Matsumoto.
\newblock The probability distributions of the first hitting times of Bessel processes.
\newblock {\em Trans.\ Amer.\ Math Soc.} \textbf{365} (10) (2013) 5237 -- 5257.

\bibitem{H2016}
S. C. Harris, M. Hesse and A. E. Kyprianou.
\newblock Branching brownian motion in a strip: survival near criticality.
\newblock {\em Ann.\ Probab.} \textbf{44} (1) (2016) 235 -- 275.

\bibitem{KT1975}
S. Karlin and M. Taylor.
\newblock {\em A First Course in Stochastic Processes}. Academic Press, New York, 1975.

\bibitem{K1978}
H. Kesten.
\newblock Branching Brownian motion with absorption.
\newblock {\em Stoch.\ Proc.\ Appl.} \textbf{7} (1) (1978) 9 -- 47.

\bibitem{OCE2017}
M. \"{O}z, M. \c{C}a\u{g}lar and J. Engl\"{a}nder.
\newblock Conditional speed of branching Brownian motion, skeleton decomposition and application to random obstacles. 
\newblock{\em Ann.\ Inst.\ H.\ Poincar\'e Probab.\ Statis.} \textbf{53} (2) (2017) 842 -- 864.

\bibitem{O2020}
M. \"Oz.
\newblock Large deviations for local mass of branching Brownian motion.
\newblock {\em ALEA Lat.\ Am.\ J.\ Probab.\ Math.\ Stat.} \textbf{17} (2020) 711 -- 731.

\bibitem{O2023}
M. \"Oz.
\newblock {\em Branching Brownian motion in an expanding ball and application to the mild obstacle problem}. arXiv:2106.00575 

%\bibitem{O2024}
%M. \"Oz.
%\newblock {\em Branching Brownian motion under soft killing}. arXiv:2307.02957 

\bibitem{S1958}
B. A. Sevast’yanov.
\newblock Branching stochastic processes for particles diffusing in a bounded domain with absorbing boundaries.
\newblock {\em Theor.\ Probab.\ Appl.} \textbf{3} (2) (1958) 121 -- 136.
 

\end{thebibliography}

\end{document}